\documentclass{amsart}
\usepackage{amsmath,amssymb,hyperref,mathrsfs,graphicx}
\newtheorem{The}{Theorem}[section]

\newtheorem{Lem}[The]{Lemma}

\theoremstyle{definition}
\newtheorem{defn}[The]{Definition}
\theoremstyle{remark}
\newtheorem{Rem}[The]{Remark}

\numberwithin{equation}{section}
\newcommand{\T}{\mathbb{T}}

\title[a non-differentiable $\alpha$ function]{Non-differentiability of $\alpha$ function at the boundary of flat}
\author{Jianlu Zhang}
\address{Department of Mathematics, Nanjing University\\Nanjing, China, 210093}
\email{jellychung1987@gmail.com}
\thanks{}
\date{}
\begin{document}
\maketitle

\begin{abstract}
With the variational method introduced by J Mather, we construct a mechanical Hamiltonian system whose $\alpha$ function has a flat $\mathbb{F}$ and is non-differentiable at the boundary $\partial\mathbb{F}$. In the case of two degrees of freedom, we prove this phenomenon is stable under perturbations of Ma\~{n}\'{e}'s .
\end{abstract}

\section{Introduction}

Let $M$ be a smooth closed manifold with $TM$ as tangent bundle. We call such a function $L(x,v)\in C^r(TM,\mathbb{R})\;(r\geq 2)$ $\mathbf{Tonelli}$ $\mathbf{Lagrangian}$ if it satisfies:
\begin{itemize}
\item\textbf{convexity} : For all $x\in M$, $v\in T_x M$ the Hessian matrix $\frac{\partial^2 L}{\partial v_i\partial v_j}(x,v)$ is positive definiteness;
\item\textbf{superlinearity} : $\lim\limits_{\parallel v\parallel\longrightarrow\infty}\frac{L(x,v)}{\parallel v\parallel}=\infty$ uniformly on $(x,v)\in TM$;
\item\textbf{completeness}: All solutions of the corresponding Euler-Lagrangian equation are well defined for $t\in\mathbb{R}$.
\end{itemize}
Here the Euler-Lagrangian equation is given by:
\begin{equation*}
  \quad\quad\quad\quad\quad\quad\frac{d}{dt}\frac{\partial L}{\partial v}(x,v)=\frac{\partial L}{\partial x}(x,v),\quad(x,v)\in TM.\quad\quad\quad\quad\quad\quad\text{(E-L)}
\end{equation*}

\begin{Rem}
In the autonomous case, the completeness is natural under the first two assumptions. That's because we can get the Hamiltonian as
\[
H(x,v)=\frac{\partial L}{\partial v}(x,v)-L.
\]
From \cite{Carneiro1995} we know that along each orbit $(\gamma,\dot{\gamma})$ of Euler-Lagrangian equation $H(\gamma,\dot{\gamma})$ is constant. The superlinearity implies that the level sets of Hamiltonian are compact. This in turn assure the completeness of flow.
\end{Rem}
Usually we take $M=\mathbb{T}^n$. Adding a closed 1-form $\eta_c$ with the cohomology class $[\eta_c]=c\in H^1(M, \mathbb{R})$, we get a new Tonelli Lagrangian $L-\eta_c$, denoted by $L-c$ for short. From \cite{Mather1991} we know that the E-L flow $(\gamma,\dot{\gamma}) $ of $L-\eta_c$ also satisfies the E-L equation of $L$. So we can define a c-minimal curve $\gamma\in C^1(\mathbb{R},M)$ if it satisfies:
\[
\mathbb{A}_c(\gamma)=\min_{\xi(a)=\gamma(a)\atop\xi(b)=\gamma(b)}\int_a^b(L-\eta_c)(\xi(t),\dot{\xi}(t))dt,\;\forall a<b\in\mathbb{R},\;\xi\in C^{ac}(\mathbb{R},M).
\]
All of the c-minimal orbit $(\gamma,\dot{\gamma})$ form a set denoted by $\widetilde{\mathcal{G}}(c)$, which is invariant under the Euler-Lagrangian flow $\Phi^t$.

Let $\mathcal{M}_{inv}$ be the set of $\Phi^t$-invariant propability measures on $TM$. We define the $\alpha$ function as follow:
\[
\alpha(c)=-\min_{\mu\in \mathcal{M}_{inv}}\int_{TM}L-\eta_c d\mu,\quad c\in H^1(M,\mathbb{R}).
\]
As is showed in \cite{Mather1991}, there exists at least one measure $\mu_c$ such that the minimum attains. We call this measure $\mu_c$ c-minimal measure. The union of the supports of all c-minimal measures is called Mather set, denoted by $\widetilde{\mathcal{M}}(c)$.

Since we know that $\alpha(c)$ is convex, finite everywhere and superlinear \cite{Mather1991}, we can define its conjugate function in the sense of convex analysis\cite{Rockafellar} as:
\[
\beta(h)=\min_{\rho(\mu)=h\atop\mu\in\mathcal{M}_{inv}}\int_{TM}L d\mu,\quad h\in H_1(M,\mathbb{R}),
\]
here $\rho(\mu)$ is defined via the De Rham inner product:
\[
\langle\rho(\mu),c\rangle\doteq\int\eta_c d\mu.
\]
$\beta(h)$ is also a convex, finite everywhere and superlinear function. From \cite{Mane1996} and \cite{Zheng} we can get the following properties:
\begin{The}
  \begin{itemize}
    \item If $\mu$ is a c-minimizing measure, we have $\rho(\mu)\in D^{-}\alpha(c)$.
    \item The maximal connected domain on which $\alpha$ function isn't strict convex is called a flat $\mathbb{F}$. $\forall c,c'\in int\mathbb{F}$ we have $\widetilde{\mathcal{M}}(c)=\widetilde{\mathcal{M}}(c')$.
    \item For each non-differential point c of $\alpha$ function, $\widetilde{\mathcal{M}}(c)$ corresponds to at least two ergodic components with different rotation vectors.
    \item If $h$ is a strict convex point of $\beta$ function, then there must exist one ergodic minimal measure $\mu$ with $\rho(\mu)=h$.
  \end{itemize}
\end{The}

We also need to define another two sets called Aubry set $\widetilde{\mathcal{A}}(c)$ and Ma\~{n}\'{e} set $\widetilde{\mathcal{N}}(c)$. First we define:
\[
h_c^t(x,y)\doteq\min_{\gamma\in C^1([0,t],M)\atop\gamma(0)=x,\gamma(t)=y}\int_0^t (L-\eta_c)(\gamma(s),\dot{\gamma}(s))+\alpha(c)ds,
\]
\[
\Phi_c(x,y)\doteq\inf_{t\in\mathbb{R}}h_c^t(x,y),
\]
and
\[
h_c^{\infty}(x,y)\doteq\liminf_{t\rightarrow+\infty}h_c^t(x,y).
\]
We call a curve $\gamma\in C^1(\mathbb{R},M)$ c-semi-static if
\[
\Phi_c(\gamma(a),\gamma(b))=\int_a^b(L-\eta_c)(\gamma(t),\dot{\gamma}(t))+\alpha(c)dt,\quad\forall a<b\in\mathbb{R},
\]
and a curve c-static if
\[
\Phi_c(\gamma(b),\gamma(a))+\int_a^b(L-\eta_c)(\gamma(t),\dot{\gamma}(t))+\alpha(c)dt=0,\quad\forall a<b\in\mathbb{R}.
\]
A c-static orbit must be a c-semistatic orbit \cite{Bernard P}. We call the union of all global c-semistatic orbits Ma\~{n}\'{e} set denoted by $\widetilde{\mathcal{N}}(c)$, and the union of all global c-static orbits Aubry set denoted by $\widetilde{\mathcal{A}}(c)$. From \cite{Bernard P} we have the following inclusions:
\begin{The}\label{Alpha}
\begin{itemize}
  \item $\widetilde{\mathcal{M}}(c)\subseteq\widetilde{\mathcal{A}}(c)\subseteq\widetilde{\mathcal{N}}(c)\subseteq\widetilde{\mathcal{G}}(c)$.
  \item We denote the projected set of $\widetilde{\mathcal{M}}(c)$ and $\widetilde{\mathcal{A}}(c)$ from $TM$ to $M$ by $\mathcal{M}(c)$ and $\mathcal{A}(c)$. Then the inverse of the project map $\pi$  \[\pi^{-1}:\mathcal{A}(c)\rightarrow\widetilde{\mathcal{A}}(c)\] is a Lipschitz graph.
  \item For an autonomous Lagrangian we have $\widetilde{\mathcal{N}}(c)=\widetilde{\mathcal{G}}(c)$. $\widetilde{\mathcal{N}}(c)$ is an upper-semicontinuous set-valued function of $c$.
  \item If c-minimal measure is uniquely ergodic, then we have $\widetilde{\mathcal{N}}(c)=\widetilde{\mathcal{A}}(c)$.
  \item $\forall c,c'\in int\mathbb{F}$ we have $\widetilde{\mathcal{A}}(c)=\widetilde{\mathcal{A}}(c')$, and $\widetilde{\mathcal{A}}(int\mathbb{F})\subseteq\widetilde{\mathcal{A}}(\partial\mathbb{F})$.
\end{itemize}
\end{The}
In \cite{Mane1996}, Ma\~{n}\'{e} raised the problem: Is it true that for a generic Lagrangian, it's $\alpha$ function is of $C^1$ smoothness? Here the `generic' means there is a residue subset $\mathcal{G}\subset C^{\infty}(M,\mathbb{R})$ such that the property holds for all the Tonelli Lagrangian $L-u,u\in\mathcal{O}$. Negative answer is firstly given in \cite{Saghin}, but we don't know whether there would be a flat $\mathbb{F}$ coming out, not to mention the stability.

Moreover, we know $int\mathbb{F}$ share the same Mather set but there may be new measure coming out at $c\in\partial\mathbb{F}$. If so, we can construct heteroclinic orbits between these different measures and explore much interesting dynamic behavior. This phenomenon was discovered by Zheng Yong in \cite{Zheng}. Based on such a premise:
\[
\exists c\in\partial\mathbb{F},\;s.t.\inf_{g\in H_1(M,\mathcal{N}(c^*),\mathbb{Z})\atop c^8\in int\mathbb{F}}h_c^{\infty}(g)=\delta>0,
\]
\[
h_c^{\infty}(g)\doteq\inf_{x\in\bar{M}}h_c^{\infty}(x,x+g).
\]
Here $\bar{M}$ is the universal cover space of $M$. He gets the following conclusion in $\bar{M}$.
\begin{The}\cite{Zheng}
  There exist infinitely many $\bar{M}$-minimal homoclinic orbits to $\widetilde{\mathcal{A}}(int\mathbb{F})$ which are not of multi-bump type.
\end{The}

Actually, under this premise, we can assure one new minimal measure's coming out at $\partial\mathbb{F}$ (sec.2, lemma2.1 in \cite{Zheng}). This new measure forms a mechanism to construction these $\bar{M}$-minimal homoclinic orbits. But he didn't give the answer whether if there exists one system satisfies his premise. Our construction gives a posititve answer and verifies the rationality of the premise. Also we can use his Theorem to get infinitely many $\bar{M}$-minimal homoclinic orbits for our example.


In this paper we can construct the following example:
\begin{equation}\label{1.1}
L(x,\dot{x})=\langle
\dot{x},\dot{x}\rangle_g+u_1(x_n)+u_2(x_n)\sum_{i=1}^{n-1}(1-\cos
x_i),
\end{equation}
here the precise form of the Riemannian metric $g$ and the potential function will be given in the later section. Then we get our main result as following:

\begin{The}\label{main theorem}
In the case of $n$ degrees of freedom, there exist $u_1$,$u_2$ and Riemannian metric $g$ such that the $\alpha$ function of (\ref{1.1}) has a flat $\mathbb{F}$ of full dimension at the lowest energy level. There are  $2^{n-1}$ points $c_i(i=1,2,\cdots,2^{n-1})\in\partial\mathbb{F}$ of which the Mather set $\mathcal{M}(c_i)$ supports at $n-1$ ergodic minimizing measures except the
one supported on the minimal fixed point. Besides, all of the $n-1$ new minimizing measures have non-vanishing rotation vectors different from each other.
\end{The}
\begin{Rem}
We can see that in the following section this phenomenon also happens at a higher energy level flat.
\end{Rem}
\begin{defn}
  We say a property of $L$ is stable in the sense of Ma\~{n}\'{e} if there exists a small enough neighborhood $\mathcal{O}$ of zero in $C^{\infty}(M,\mathbb{R})$ such that the property holds for all the Tonelli Lagrangian $L-u,u\in\mathcal{O}$.
\end{defn}
\begin{The}\label{two dimensions}
 In the case of two degrees of freedom, the properties in Theorem(\ref{main theorem}) is stable in the sense of Ma\~{n}\'{e}.
\end{The}

This two degrees of freedom phenomenon can be applied to a-priori stable Arnold Diffusion problem\cite{Cheng C-Q}. Recently, Cheng proved the existence of normally hyperbolic invariant cylinder near double resonance frequency (\cite{Cheng C-Q} sec.5.2), which can be divided into two different cases. Actually, our example just satisfies the case that normally hyperbolic invariant cylinder can reach the lowest energy level with a periodic orbit as the bottom. We also know this case is stable of Ma\~{n}\'{e}'s sense (Theorem\ref{two dimensions}).

Similar construction can be found in \cite{Bangert V} from a geometrical viewpoint. Here we use a variantional method which is known as Aubry Mather Theory nowadays and verify the existence of $\alpha$ function's flat. Besides, we supply Zheng Yong's paper and get infinitely many homoclinic orbits. So our construction has extra significance in dynamic systems other than geometry.

This paper is outlined as follows. In Section $2$, we give the construction of the examples of different cases. We give the proof of stability in the case of two degrees of freedom and some remarks in Section $3$.

%
%
%
%
\section{ Construction of the example}
Now we construct the example in details and verify the properties satisfied in Theorem \ref{main theorem}.
\subsection{the lowest energy level case}

We construct the Lagrangian as follows.
\[
L(x,\dot{x})=\langle
\dot{x},\dot{x}\rangle_g+u_1(x_n)+u_2(x_n)\sum_{i=1}^{n-1}(1-\cos
x_i)
\]
where $(x,\dot{x})\in \mathbb{T}^n\times\mathbb{R}^n$ and $x=(x_1,x_2,\ldots,x_n)$. $u_1(x_n)$, $u_2(x_n)$ and the Riemannian metric $g$ are constructed as follows.

First, we mark several channels on the $\mathbb{T}^n$ along the $n$-th coordinate $x_n$.
\begin{enumerate}
  \item channel $A$: $\mathbb{T}^{n-1}\times[-\frac{\pi}{4n},\frac{\pi}{4n}],\;\mathbb{T}=\mathbb{R}/2\pi$,
  \item channel $B_i$: $\mathbb{T}^{n-1}\times[\frac{2\pi i}{n}-\frac{\pi}{4n},\frac{2\pi i}{n}+\frac{\pi}{4n}],\;i=1,2,\cdots n-1$,
  \item channel $C_j$: $\mathbb{T}^{n-1}\times[\frac{(2j-1)\pi}{n}-\frac{\pi}{4n},\frac{(2j-1)\pi}{n}+\frac{\pi}{4n}],\;j=1,2,\cdots n$.
\end{enumerate}
Obviously, channel $B_i$ and $C_j$ are just constructed by shifting channel $A$ along the $x_n$ coordinate. It is easy to see that $A$ and $B_i$ are separated by $C_j$.
\begin{equation}
u_1(x_n)=\begin{cases}
0,\quad x_n\ \text{in channel}\ A,\\
K\gg 1,\quad x_n\ \text{in channel}\ C_j\ \text{and}\ j=1,2,\cdots n,\\
0<\delta_i\ll K,\quad x_n\  \text{in channel}\ B_i\ \text{and}\ i=1,2,\cdots n-1,\\
\end{cases}
\end{equation}
where the values of $\delta_i$ will be given later (in fact it is enough to take $0<\delta_i\leq\frac{1}{2}$).
\begin{equation}
u_2(x_n)=\begin{cases}
1,\quad x_n\ \text{in channel}\ A,\\
0,\quad x_n\ \text{out of some small neighborhood of channel A}.\\
\end{cases}
\end{equation}
Both $u_1$ and $u_2$ can be smoothly extended to the whole space $\mathbb{T}^n$. Later you will see that we needn't give their precise evaluation because only qualitative proof is cared.

At last, we construct the Riemannian metric
$\langle\dot{x},\dot{x}\rangle_g=\sum_{i=1}^{n}a_i(x)^2\dot{x}_i^2$ as follows:
\begin{equation*}
G
 =\left(\begin{array}{cccc}
 a_1^2(x)&0&\cdots&0\\
 0&a_2^2(x)&0&0\\
 0&0&\ddots&0\\
0&0&0&a_n^2(x)
\end{array}\right)_{n\times n}
\end{equation*}
is a diagonal metric matrix.
\begin{enumerate}
  \item $G=\mathbf{Id}_{n\times n}$, for $x_n$ in channel $A$,
  \item \begin{equation*}
             G=\left(\begin{array}{cccc}
               K&0&\cdots&0\\
               0&K&0&0\\
               0&0&\ddots&0\\
                0&0&0&K
               \end{array}\right)_{n\times n}
               \end{equation*} for $x_n$ in channel $C_j(j=1,2,\ldots,n)$,
  \item \begin{equation*}
          G=\left(\begin{array}{cccccccc}
            \frac{1}{8}&0&\cdots&\cdots&\cdots&\cdots&\cdots&0\\
            0&\frac{1}{8}&0&\cdots&\cdots&\cdots&\cdots&0\\
            0&0&\ddots&0&\cdots&\cdots&\cdots&0\\
            0&\cdots&0&\frac{1}{8}&0&\cdots&\cdots&0\\
            0&\cdots&\cdots&0&\frac{1}{4}&0&\cdots&0\\
            0&\cdots&\cdots&\cdots&0&\frac{1}{8}&0&0\\
            0&\cdots&\cdots&\cdots&\cdots&0&\ddots&0\\
            0&\cdots&\cdots&\cdots&\cdots&\cdots&0&\frac{1}{8}
            \end{array}\right)_{n\times n}
            \end{equation*} for $x_n$ in channel $B_k(k=1,2,\ldots,n-1)$, where $\frac{1}{4}$ is in the $k$-th row and $k$-th column.
\end{enumerate}

We take a subspace of $H^1(\mathbb{T}^{n},\mathbb{R})$ with the n-th coordinate $c_n=0$, which is denoted by $\mathcal{H}$. We also denote the restriction of Lagrangian L in the channels $A$, $B_i$ and $C_j$ by $L_A$, $L_{B_i}$ and $L_{C_j}$, $\alpha_A(c)$, $\alpha_{B_i}(c)$ and $\alpha_{C_j}(c)$ as the restricted $\alpha$ function. It's remarkable that $\alpha_A(c)$, $\alpha_{B_i}(c)$ and $\alpha_{C_j}(c)$ are well defined for all the $c\in H^1(\mathbb{T}^{n},\mathbb{R})$ but only $c\in\mathcal{H}$ we care about.

Moreover, we could calculate $\alpha_A(c)$, $\alpha_{B_i}(c)$ and $\alpha_{C_j}(c)$ when $c\in\mathcal{H}$.
\begin{align*}
\alpha_A(c)&=-\min_{\mu\in\mathcal{M}_{inv}}\int L_A-c d\mu,\\
&=-\min_{\mu\in\mathcal{M}_{inv}}\int \langle v,v\rangle_g+\sum_{i=1}^{n-1}(1-\cos x_i)-\sum_{i=1}^{n-1}c_iv_i d\mu,\\
&=-\min_{\mu\in\mathcal{M}_{inv}\atop\pi_{n}\rho(\mu)=0}\sum_{i=1}^{n-1}\int(v_i^2+(1-\cos x_i)-c_iv_i) d\mu.\\
\end{align*}
It is easy to see that $L_A$ consists of $n-1$ independent pendulum systems. Hence, $\alpha_A$ has a flat  $[-\frac{4\sqrt{2}}{\pi},\frac{4\sqrt{2}}{\pi}]^{n-1}\times\{0\}\subseteq \mathcal{H}$.
\begin{align*}
\alpha_{B_i}(c)&=-\min_{\mu\in\mathcal{M}_{inv}}\int L_{B_i}-c d\mu\\
&=-\min_{\mu\in\mathcal{M}_{inv}}\int \langle v,v\rangle_g+\delta_i-\sum_{i=1}^{n-1}c_iv_i d\mu\\
&=-\delta_i+\sum_{k=1}^{n-1}16c_k^2-12c_i^2,\quad i=1,2,\cdots,n-1.\\
\end{align*}
We get the last equality because this is a geodesic integral system. If we take $0<\delta_i<\frac{1}{2}$, the level set $\alpha_{B_i}^{-1}(0)\ \subseteq [-\frac{4\sqrt{2}}{\pi},\frac{4\sqrt{2}}{\pi}]^{n-1}\times\{0\}\subseteq \mathcal{H}$.

The following lemma gives the relationship between $\alpha$ function on $\T^n$ and $\alpha$ function restricted to channel $A$ and $B_i$.

\begin{Lem}\label{lemma 2.1}
\[\alpha_L(c)=\max\limits_{i=1,2,\cdots,n-1}\{\alpha_A,\alpha_{B_i}\},\]where
$c=(c_1,c_2,\cdots,c_{n-1},0)\in\mathcal{H}$. This result is stable in the sense of Ma\~{n}\'{e}.
\end{Lem}

\begin{proof}First, we have $\alpha_L(c)\geq \alpha_A(c)$, and $\alpha_L(c)\geq \alpha_{B_i}(c)$. It follows since the action of the local minimizer is bigger than the global minimizer. Hence,
\[
 \alpha_L(c)\geq\max_{i=1,2,\cdots,n-1}\{\alpha_A,\alpha_{B_i}\}.
\]

On the other hand, from \cite{Mane1996} we know that every minimizing measure can be approximated by a sequence of propability measures supported on absolutely continuous closed curves which are not necessarily minimal, so we just need to certify that every closed curve $\gamma$ in $\mathbb{T}^n$ has a larger action of the Lagrangian $L-c$ than its projection $\gamma^*$ to $\mathbb{T}^{n-1}$ in some channel $A$ or $B_i$, here $c=(c_1,c_2,\cdots,c_{n-1},0)\in\mathcal{H}$. It follows from
\begin{align*}
\int_\gamma L-c^*- \int_{\gamma^*}L-c^*&\geq \int a_n^2\dot{\gamma}_n^2+u_1(x_n)\\
&\geq0.
\end{align*}
Hence,
\[
\alpha_L(c)\leq\max\limits_{i=1,2,\cdots,n-1}\{\alpha_A,\alpha_{B_i}\}.
\]

Now we prove the stability. If another Lagrangian $L'$ is $C^1$ close to $L$, i.e. $L'=L+V(x)$ and $\|V\|_r\leq\epsilon\ll\delta_i$,\;$i=1,2,\cdots,n-1$. Now the values of $u_1$ and $u_2$ in channels $A$, $B_i$ and $C_j$ will be deformed with a magnitude not bigger than $\epsilon$. Without loss of generality, we can denote these channels' $\epsilon$ neighborhoods by $A_{\epsilon}$, $B_{i,\epsilon}$ and $C_{j,\epsilon}$ and avoid the influence of deformation. For convenience we still use $A$, $B_i$ and $C_j$ without confusion.

We claim that for an arbitrary closed curve $\gamma$, we could find $\xi$ in channel A or $B_i$ closed, and has a smaller action, i.e. $\mathbb{A}_{c,L'}(\xi)\leq \mathbb{A}_{c,L'}(\gamma)$. Once this claim proved, the stability we get.

$\mathbf{Case\;1.}$ Once the closed curve $\gamma$ in channel $A$, we have $\mathbb{A}_{c,L'}(\gamma)=\mathbb{A}_{c,L'}\mid_{A}(\gamma)$.

$\mathbf{Case\;2.}$ If the $n$-th component of the  homology type of $\gamma$ is not zero, it must cross $A$ and at least one $B_{i}$, then we project the part outside of channel $A$ to channel $A$ with the $n-th$ coordinate keeping constant, which is denoted by $(\gamma\cap A^c)^*$ and:
\begin{align*}
\int_{\gamma\cap A^c} L'-c^*- &\int_{(\gamma\cap A^c)^*}L'-c^*\\
&\geq \int_{\gamma\cap A^c} K_1\dot{\gamma}_n^2+K_2\\
&\geq K_2(T_2-T_1)+\frac{K_1}{T_2-T_1}\int_{T_1}^{T_2}\dot{\gamma}_n^2 dt\int_{T_1}^{T_2}1 dt\\
&\geq K_2(T_2-T_1)+\frac{K_1}{T_2-T_1}(\int_{T_1}^{T_2}\mid\dot{\gamma}_n\mid dt)^2\\
&\geq 2\sqrt{K_1K_2}\cdot width(C_{i})\\
&>0,
\end{align*}
here $K_1$ and $K_2$ could be both positive and sufficient large. In fact, $K_1=K^2$ and $K_2=K-\epsilon\gg\frac{K}{2}$ if we take $\epsilon$ sufficiently small. The width of some channel $C_{i}$ is here because $A$ and $B_{i}$ are separated by at least one $C_{i}$.
\begin{figure}[htp]
 \centering
  \includegraphics[width=3.5cm,height=4cm]{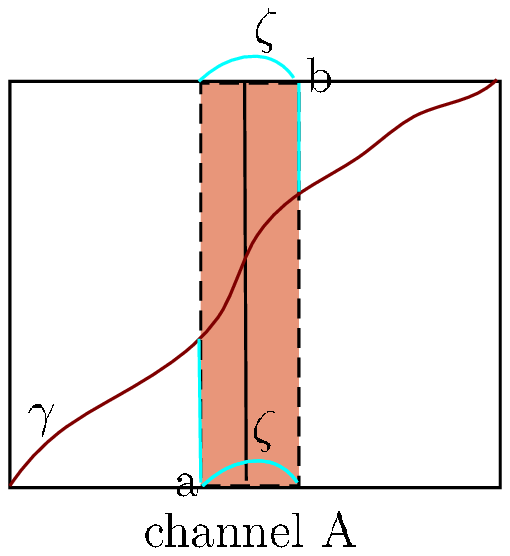}
  \caption{}
  \label{zfig5}
\end{figure}
We denote the two endpoints of $(\gamma\cap A^c)^{*}$ by $a$, $b$. Obviously these two points are both in channel $A$. Then we connect these two points with a geodesic curve $\zeta$ and $(\gamma\cap A^c)^{*}\cup\zeta\cup(\gamma\cap A)\doteq\xi$ forms a closed curve (see figure(\ref{zfig5})). But we know $dist(a,b)\leq\sqrt{n}$, and
\begin{align*}
\int_0^1 L'(\zeta,\dot{\zeta})-\langle c^{*},\dot{\zeta}\rangle dt&\leq\epsilon+1+\int_0^1|\dot{\zeta}|^2dt\\
\leq n+1+\epsilon.
\end{align*}

Then we could take sufficient large $K$ such that $n+1+\epsilon\leq 2\sqrt{K_1K_2}\cdot width(C_i)$. So $\xi$ has smaller action than $\gamma$.


With a similar approach we can also project $\gamma$ into each channel $B_i$ with the $n$-th coordinate constant and make it closed. Since all the possible cases are finitely many, we can always take a sufficient large $K$ to make all these cases satisfy the claim. Therefore, the lemma is proved.
\end{proof}

\begin{Rem}
The stability property under perturbation in this lemma will be used in the next section.
\end{Rem}

Now, we have constructed our first example of $n$ degrees of freedom. $\alpha_L^{-1}(0)$ is a lowest flat of $n-1$ dimensions restricted to $\mathcal{H}$. This is because $\alpha_{B_i}^{-1}(0)\subseteq[-\frac{4\sqrt{2}}{\pi},\frac{4\sqrt{2}}{\pi}]^{n-1}\times\{0\}\subseteq \mathcal{H}$. At the boundary of this flat, we could find $2^{n-2}$ pairs of points which are diametrical with each other. We can see that at these points $\alpha_L(c)$ is not differentiable along at least $n-1$ different directions, since $\alpha_{B_i}(c)$ intersects the flat $[-\frac{4\sqrt{2}}{\pi},\frac{4\sqrt{2}}{\pi}]^{n-1}\times\{0\}$ transversally and $\alpha_L(c)=\max_{i=1,2,\cdots,n-1}\{\alpha_A,\alpha_{B_i}\}$.
More precisely, the coordinates of these points in $\mathcal{H}$ are easy to write down when we take $\delta_i\equiv\frac{1}{2}$:
\[
\{(c_1,c_2,\cdots,c_{n-1},0)\in H^1(\mathbb{T}^n,\mathbb{R})\mid c_i=\pm\frac{1}{\sqrt{8(4n-7)}}\}
\]
At each point, there are $n-1$ new measures coming out except the one supported on the fix point in channel $A$. Also we can see that the extra $n-1$ new measures are all Lagrangian invariant torus consist of periodic orbits, but we can change the Riemannian metric in the channel $B_i$ to make all the new measures into quasi-periodic KAM torus. What we just need to do is change the metric vector $(a_1,a_2,\cdots,a_n)$ to non-resonant one.

At last, we show that the lowest level flat $\mathbb{F}$ is indeed of full dimension. That's because
\[
\widetilde{\mathcal{M}}(0)=\widetilde{\mathcal{N}}(0),
\]
and we also knows that $\widetilde{\mathcal{M}}(0)$ just consists of the fixed points (simple calculation of a pendulum system). So we could make use of the upper semi-continuous property of the $\widetilde{\mathcal{N}}(c)$ and expand the flat along the directions of $H_1(M,\mathcal{N}(0),\mathbb{R})$.
\begin{Lem}
If $c\in\mathcal{H}^{\perp}\subseteq H_1(\mathbb{T}^n,\mathbb{R})$, there exists a sufficient small  $ \lambda\ll1$, such that there exists a flat $[-\lambda c,\lambda c]$ along the direction of $c$.
\end{Lem}
\begin{proof}
By the upper semi-continuity of set-valued function $c\longrightarrow\widetilde{\mathcal{N}}(c)$, we know for sufficiently small $\lambda$, $\mathcal{N}(\lambda c)$ is still in channel A because $\mathcal{M}(0)=\mathcal{A}(0)=\mathcal{N}(0)=\{0\}\in\mathbb{R}^n$. Then:
\[
-\alpha(\lambda c)=\int L-\lambda cd\mu_{\lambda c}=\int Ld\mu_{\lambda c}-\langle \lambda c, [\mu_{\lambda c}]\rangle=\int Ld\mu_{\lambda c}\geq\int Ld\mu_{0}=-\alpha(0),
\]
so we have $\alpha(\pm \lambda c)\leq\alpha(0)$, where the $-\lambda c$ case is the same with $\lambda c$. We get $\alpha(\pm\lambda c)=\alpha(0)$ because of the convexity, and then $\alpha_L(c)$ has a full dimensional flat at the lowest energy level.
\end{proof}
We also recall that $\forall h\in\mathbb{R}$ the sublevel set $\{c| \alpha(c)\leq h\}$ is a convex set. Using the previous Lemma we get a full dimensional flat at the lowest level.
\subsection{the higher energy level case}

Now, we show that the same phenomenon can happen at a higher energy level.

For the sake of simplicity, we set $n=3$ and it's easy to generalize the example to a higher dimensional case$(n\geq3)$. We could take the Lagrangian as:
\[
L(x,\dot{x})=\frac{1}{2}a_1(x)(\dot{x}_1-1)^2+\frac{1}{2}a_2(x)\dot{x}_2^2+\frac{1}{2}a_3(x)\dot{x}_3^2+u_1(x_3)+u_2(x_3)(1-cos x_2)
\]
We just need to set two C-type channels , one A-type channel and one B-type channel. The settings of $u_1$ and $u_2$ is the same as the previous section. Since we still restrict the cohomology to $\mathcal{H}$, so we have $\dot{x}_3\equiv0$ and:
\[
L_{C_i}\geq \frac{1}{2}(\dot{x}_1-1)^2+{K},\quad i=1,2,
\]
\[
L_B=\frac{1}{2}(\dot{x}_1-1)^2+\dot{x}_2^2+\delta,
\]
and
\[
L_A=\frac{1}{2}(\dot{x}_1-1)^2+\frac{1}{2}\dot{x}_2^2+(1-\cos x_2).
\]
Then Lemma \ref{lemma 2.1} is still valid:

\[\alpha_L(c)\big|_{c_3=0}=\max\limits\{\alpha_A,\alpha_{B}\}.\]
Based on our calculation and $c_3=0$, the set $\{\alpha_A=0\}$ is $\{0\}\times[-\frac{4\sqrt{2}}{\pi},\frac{4\sqrt{2}}{\pi}]\times\{0\}$, which is a flat of one dimension. And $\{\alpha_B=0\}$ is an elliptical curve with an expression as
\[
\frac{1}{2}(c_1+1)^2+\frac{1}{4}c_2^2=\frac{1}{2}+\delta.
\]
Restrict to $c_1=0$ and we find that the elliptical curve go across $\{\alpha_A=0\}$ from inner with the intersection points' coordinates $(0,\pm 2\sqrt{\delta},0)$. In order to show that $\alpha_L$ is not differentiable at these two points (in fact we just need show that for one point), we need to calculate the directional derivative along $\overrightarrow{e}_2$.
\[
\partial_{\vec{{e}}_2}^+\alpha_L(0,2\sqrt{\delta},0)=\partial_{\vec{e}_2}\alpha_B(0,2\sqrt{\delta},0)
=\sqrt{\delta},
\]
and on the other hand
\[
\partial_{\vec{e}_2}^-\alpha_L(0,2\sqrt{\delta},0)=\partial_{\vec{e}_2}\alpha_A(0,2\sqrt{\delta},0)
=0.
\]
So $\alpha_L(0,2\sqrt{\delta},0)$ is not differentiable along the direction $\vec{e}_2$, which means at this point $(0,2\sqrt{\delta},0)$, there exists an extra minimizing measure of a different homology besides the one supported on closed curve.

\begin{Rem}

With the upper semi-continuous property of the $\widetilde{\mathcal{N}}(c)$, we can see that in the case of three degrees of freedom, the flat is in fact of two dimensions. The following graph shows the exact situation of the flat.
\end{Rem}
\begin{figure}[htp]
 \centering
  \includegraphics[width=5cm,height=5cm]{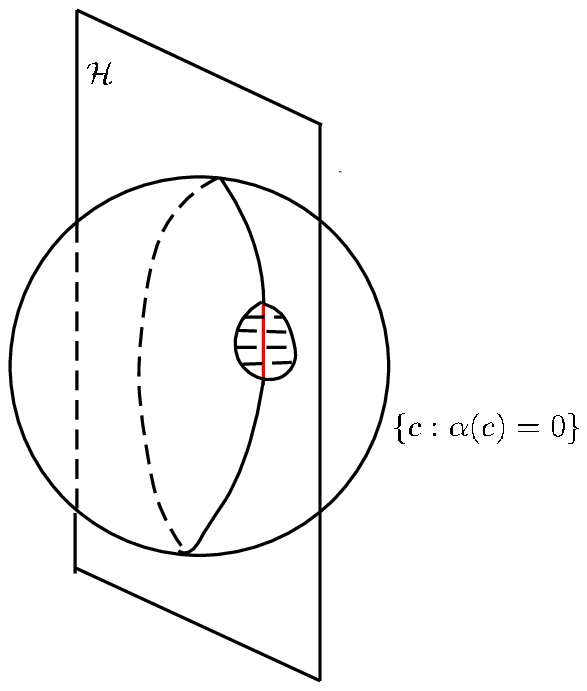}
  \caption{}
  \label{zfig3}
\end{figure}
\section{\sc Stability of the two degrees of freedom case}
In this section we prove Theorem \ref{two dimensions} in the case of two degrees of freedom.
\begin{proof}
We take $\epsilon\ll\tau^2\ll\delta\ll1$, and shrink the width of channel $A$ and $B$ to the order of $O(\epsilon)$ for the sake of simplity, and $\tau$ is used to control the value of potential function perturbation. Later we will evaluate them precisely.

A neighborhood of $0$ in $C^r(M,\mathbb{R})$ with a radius $\epsilon$ is denoted by $B(0,\epsilon)$. Any perturbation function $V(x)$ in it can't break the intersection property of $\alpha_A$ and $\alpha_B$ but just deform them of a $\epsilon$ order change (see the following graphs). This is because the continuity of $\alpha$-function with respect to Lagrangian function\cite{Zheng}.
\begin{figure}[htp]
 \centering
  \includegraphics[width=12cm,height=2.7cm]{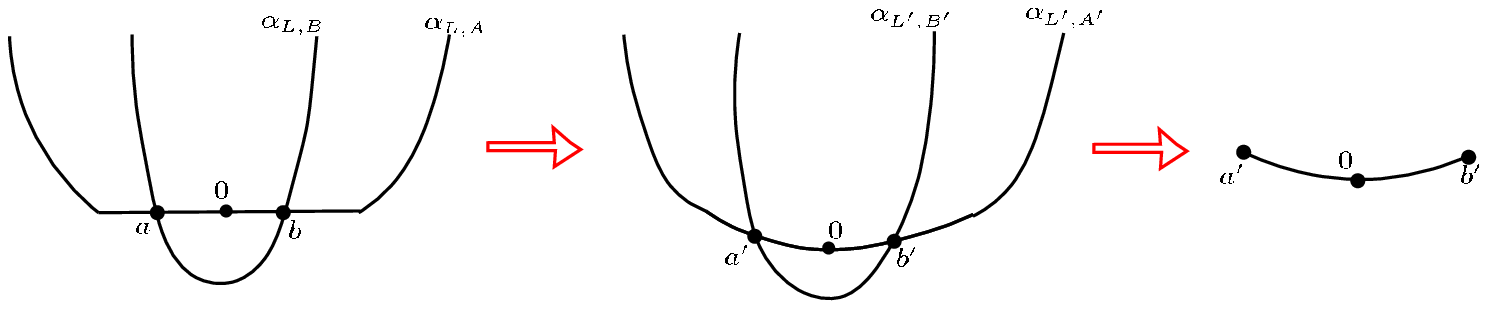}
  \caption{}
  \label{zfig2}
\end{figure}

So in a small neighborhood of point $a$, there is still an intersection point $a'$ of $\alpha_{L',A}$ and $\alpha_{L',B}$, here $L'=L+V$. We just need to deal with $a'$ point and get the same conclusion of $b'$, this is because the Hamiltonian $H'$ conjugated to $L'$ is still a quadratic mechanical system with $\alpha$ function axial symmetry. As is known to us, the Ma\~{n}\'{e} critical value is equal to $\alpha'_{L',A}(0)$, which is also the minimizing value of $\alpha'_{L',A}$ function\cite{Carneiro1995}. Without lose of generality, we suppose this value is still $0$. If not, we can add a constant to $L+V$ without influencing the intersection property of restricted $\alpha$ functions.

Still restricted to $\mathcal{H}$, the curve $\widehat{a'o}$ may be no longer a straight line, we will consider two different cases separately in the following and prove the stability. From Lemma \ref{lemma 2.1}, we know $\widetilde{\mathcal{M}}(a')$ have at least two ergodic minimizing measures, one in channel $A$ and the other in channel $B$. We denote the one in channel A by $\mu_{a',A}$ and the other $\mu_{a',B}$. From the construction of channel A, the support of this measure can only be a periodic orbit with a rotation vector $(h_1, 0)$ because of the restriction of homology.

$\mathbf{Case\;1.}$ If $h_1=0$, we could see that the curve $\widehat{a'o}$ is in fact a straight line.
This is because
\begin{align*}
  \alpha(a')&=-\int L-a' d\mu_{a',A}=-\int Ld\mu_{a',A}\\
  &\leq\int\alpha(0),
\end{align*}
and $\alpha(0)$ is the minimal value.
But $\epsilon\ll\delta$, so $\mu_{a',B}$ couldn't have a trivial homology because of the convexity of the $\alpha_{L',B}$. So we get that not only $\widehat{a'ob'}$ is straight, but also $\alpha_{L'}$ is not differential along $\mathcal{H}$ direction. So we get the stability of this case.

$\mathbf{Case\;2.}$ If $h_1>0$(the same with $h_1<0$), we let $a'\longrightarrow a''$, here $a''$ is the first point with $\alpha_{L',A}(a'')=0$ ($a''=0$ is possible). Then we know that $[a'',0]$ is a flat and $\widetilde{\mathcal{M}}(0)\subseteq\widetilde{\mathcal{A}}(a'')$, and $\widetilde{\mathcal{A}}(a'')\setminus\widetilde{\mathcal{M}}(0)\neq\emptyset$ \cite{Zheng}. Then there must exist a homoclinic orbit or a periodic orbit in $\widetilde{\mathcal{A}}(a'')$.
\begin{figure}[htp]
 \centering
  \includegraphics[width=5cm,height=2cm]{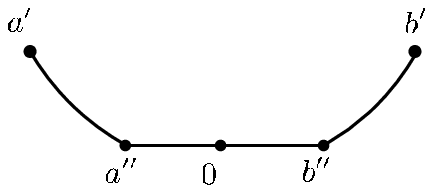}
  \caption{}
  \label{zfig1}
\end{figure}

 $\mathbf{Case\;2.1}$ If There exists a homoclinic orbit to $\widetilde{\mathcal{A}}(a'')$, because the maximal points of the potential are contained in a neighborhood of $(0,0)$ of a radius of order $O(\epsilon)$ (that's why we shrink the width of channel A,B), the minimizing homoclinic is asymptotic to this neighborhood $B(0,\epsilon)$.
\begin{figure}[htp]
 \centering
  \includegraphics[width=9cm,height=2.5cm]{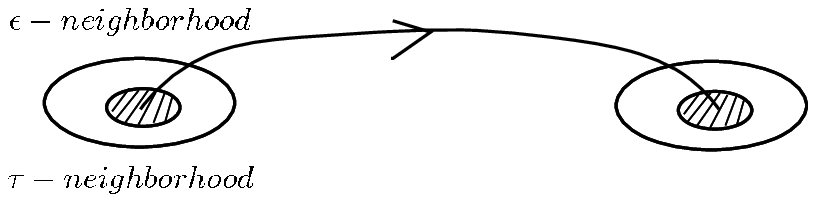}
  \caption{}
  \label{zfig1}
\end{figure}
For any $c\in[a',a'']$, $\widetilde{\mathcal{M}}(c)$ can only be made up of periodic orbits in channel A. Under the weak topology of probability measures, there exists a sequence $c_n\longrightarrow a''$ and $\mu_{L',c_n}\rightharpoonup\mu_{L',0}\in\widetilde{\mathcal{M}}(0)$. Moreover, the minimizing homoclinic orbit $\gamma$ is contained in the Hausdorff limit of $supp\{\mu_{c_n}\}$. As $\gamma$ is a static orbit in $\widetilde{\mathcal{A}}(a'')$, we have
\begin{equation}\label{equation}
\int_{-\infty}^{+\infty}L'(\gamma,\dot{\gamma})dt=\langle a'',[\gamma]\rangle,
\end{equation}
here the $\langle,\rangle$ is the inner product induced by de Rham. In fact, $\langle a'',[\gamma]\rangle$ is really a scalar product because the homology of channel A and $a''\in\mathcal{H}$. As $\epsilon\ll\tau^2\ll\delta\ll1$, we have
\begin{align}\label{align}
\int_{-\infty}^{+\infty}L'(\gamma,\dot{\gamma})dt=&\int_{-\infty}^{+\infty}\langle\dot{\gamma},\dot{\gamma}\rangle_g+u_1(\gamma_2(t))\\
&+u_2(\gamma_2(t))(1-\cos(\gamma_1(t)))+V(\gamma(t))dt\notag\\
\geq&\int_{-T}^{+T}(\dot{\gamma_1}^2+\dot{\gamma_2}^2)+u_1(\gamma_2(t))\notag\\
&+u_2(\gamma_2(t))(1-\cos(\gamma_1(t)))+V(\gamma(t))dt,\notag
\end{align}
here $\gamma\mid_{[-T,T]}$ is the part of $\gamma$ outside the $O(\tau)$ neighborhood of $(0,0)$. Recall that $\epsilon\ll\tau^2\ll\delta\ll1$, we can get a positive lower bound of potential function out of $O(\tau)$ neighborhood of $(0,0)$. As we have supposed \[
\max_{x\in A}-u_1(x)-u_2(x)(1-\cos(x))-V(x)=0\] and the width of channel A is of order $O(\epsilon)$,
 we have:
\begin{align*}
\int_{-T}^{+T}(\dot{\gamma_1}^2+\dot{\gamma_2}^2)+&u_1(\gamma_2(t))+u_2(\gamma_2(t))(1-\cos(\gamma_1(t)))+V(\gamma(t))dt\\
\geq&\int_{-T}^{+T}\dot{\gamma_1}^2+1-\cos\tau-\epsilon dt\\
\geq& 2(2\pi-2\tau)\sqrt{1-\cos\tau-\epsilon},
\end{align*}
here $(\gamma_1,\gamma_2)$ is the coordinate of $\gamma$.

At the same time, we have:
\begin{align*}
\langle a'',[\gamma]\rangle=&a''\pi_1([\gamma])=a''\\
\geq& 2(2\pi-2\tau)\sqrt{1-\cos\tau-\epsilon},
\end{align*}
because $[\gamma]\equiv(1,0)$. We already have $a''<a'$. Specially we take $\tau=\frac{1}{8}\pi$, then $2(2\pi-2\tau)\sqrt{1-\cos\tau-\epsilon}>\frac{\pi}{9}$. But the Riemannian metric can be modified in channel B to make $a'<\frac{\pi}{9}$ and there will be a contradiction to our assumption.

$\mathbf{Case\;2.2}$ For the case of periodic orbit coming out, we can give a same proof as above. We just need to modify the integral lower and upper bounds to finite in (\ref{equation}) and (\ref{align}) to get a same contradiction.
\end{proof}

\noindent\textbf{Acknowledgement} The author thanks Prof. C.Q.Cheng and W.Cheng for talking about the details, and thanks L.Wang for suggesting this problem to me.

\end{document}